\RequirePackage{ifpdf}
\ifpdf
 	\documentclass[pdftex]{amsart}
 	\else
 	\documentclass[dvips]{amsart}
\fi

\usepackage{amsfonts,amsthm,latexsym,amsmath,amssymb,amscd,amsmath, mathrsfs, epsf}
\usepackage{graphicx}

\graphicspath{{./figures/}}
\ifpdf
	\usepackage{epstopdf}		
	\DeclareGraphicsExtensions{.png,.jpg,.eps,.epsf}
\fi

\newtheorem{theorem}{Theorem}
\newtheorem{proposition}[theorem]{Proposition}
\newtheorem{lemma}[theorem]{Lemma}
\newtheorem{definition}{Definition}
\newtheorem{corollary}{Corollary}
\newtheorem{example}{Example}
\newtheorem{remark}{Remark}
\newtheorem{conjecture}{Conjecture}

 \newcommand{\la}{\lambda}

\newcommand{\Z}{\mathbb{Z}}

\newcommand{\xvec}{\mathbf{x}}
\newcommand{\yvec}{\mathbf{y}}

\newcommand{\cD}{\mathcal{D}}	

\newcommand{\PST}{\mathbf{PST}}		

\begin{document}
\title[Discriminants,  symmetrized graph monomials, and SOS]{Discriminants,  symmetrized graph monomials, and sums of squares} 

\author[P.~Alexandersson]{Per Alexandersson}

\address{ Department of Mathematics,
   Stockholm University,
   S-10691, Stockholm, Sweden}
\email{per@math.su.se}

\author[B. Shapiro]{Boris Shapiro}

\address{Department of Mathematics,
   Stockholm University,
   S-10691, Stockholm, Sweden}
\email{shapiro@math.su.se}

\begin{abstract}
Motivated by the necessities of the invariant theory of binary forms  J.~J.~Sylvester  constructed  in 1878 for  each graph with possible multiple edges but without loops its symmetrized graph monomial which is a polynomial in the vertex labels of the original graph. In the 20-th century this construction was  studied  by several authors. We  pose the question for which graphs  this polynomial is a non-negative resp. a sum of squares. 
This problem is motivated by a recent conjecture of F. Sottile and E. Mukhin on discriminant of the  derivative  
of a univariate polynomial, and an interesting example of P. and A. Lax of a graph with $4$ edges whose symmetrized graph monomial  is non-negative but not a sum of squares. We present detailed information about symmetrized graph monomials for graphs with four and six edges, obtained by computer calculations.
\end{abstract}

\maketitle

\section{Introduction} 
In what follows by a {\em graph} we will always mean a (directed or undirected) graph with (possibly) multiple edges but no loops. 
The classical construction of J.~J.~Sylvester and J.~Petersen \cite{sylvester,petersen} associates
to an arbitrary directed loopless graph a symmetric polynomial as follows:

\begin{definition}
Let $g$ be a directed graph, with vertices $x_1,\dots,x_n$ and adjacency matrix $(a_{ij}).$
Define first its graph monomial $P_g$ as follows 
$$P_g(x_1,\dots,x_n):=\prod_{1\leq i , j \leq n} (x_i-x_j)^{a_{ij}},$$
where $a_{ij}$ is the number of edges joining $x_i$ with $x_j.$

The symmetrized graph monomial of $g$ is defined as
$$\tilde{g}(\xvec)=\sum_{\sigma \in S_n} P_g(\sigma \xvec),\quad \xvec = x_1,\dots,x_n.$$
\end{definition}

Notice that if the original $g$ is undirected one can still define $\tilde g$ up to a sign by choosing an arbitrary orientation of its edges. 
Symmetrized graph monomials are closely related to $SL_2$-invariants and covariants and were introduced in 1870's in attempt  to find new tools in the invariant theory. 
Namely, to obtain an $SL_2$-coinvariant from a given $\tilde {g}(\xvec)$ we have to perform two standard operations. First we express the symmetric polynomial $\tilde{g}(\xvec)$ in $n$ variables in terms of the elementary symmetric functions $e_1,...,e_n$ and obtain the (inhomogeneous) polynomial $P_g(e_1,...,e_n)$. Secondly, we perform the standard homogenization of a polynomial of a given degree $d$ 
$$Q_g(e_0,e_1,...,e_n):=e_0^dP_g\left(\frac{e_1}{e_0},...,\frac{e_n}{e_0}\right).$$

The following fundamental proposition apparently goes back to A.~Cayley, see Theorem~2.4 of \cite{sabidussi2}. 

\begin{theorem}\label{th:inv}

\rm{(i)} If $g$ is a $d$-regular graph with $n$ vertices  then $Q_g(e_0,...,e_n)$ is either an $SL_2$-invariant of degree $d$ in $n$ variables or it is identically zero.

\noindent
\rm{(ii)} Conversely, if $Q(e_0,...,e_n)$ is an $SL_2$-invariant of degree $d$ and order $n$ then there exist $d$-regular graphs $g_1,...,g_r$ with $n$ vertices and integers $\la_1,...,\la_r$ such that 
$$Q=\la_1 Q_{g_1}+...+\la_rQ_{g_r}.$$
\end{theorem}

\begin{remark} Recall that a graph is called $d$-regular if every its vertex has valency $d$. Notice that if $g$ is an arbitrary graph then it is natural to interpret its polynomial $Q_g(e_0,...,e_n)$  as the  $SL_2$-coinvariant.) 
\end{remark} 

The natural question about the kernel of the map sending $g$ to $\tilde {g}(\xvec)$ (or to $Q_g$)   
was already discussed by J.~Petersen who claimed that he has found  a necessary and sufficient condition when $g$ belongs to the kernel, see \cite{sabidussi2}.  This claim turned out to be false. (An interesting correspondence between J.~J.~Sylvester, D.~Hilbert and F.~Klein related to this topic can be found in \cite{sabidussi}.)  
The kernel of this map seems to be related to  several open problems such as Alon-Tarsi \cite {AT} and the Rota basis conjecture \cite{W}.  
(We want to thank Professor A.~Abdesselam for this valuable information, see \cite{Abd}.)

In the present paper we are interested in examples of graphs whose symmetrized graph monomial
are non-negative resp. sum of squares. Our interest in this matter has two sources. 

The first one is a recent conjecture by F. Sottile and E. Mukhin formulated on the AIM meeting 'Algebraic systems with only real solutions'  in 
October 2010.

\begin{conjecture}\label{conj:orig} The discriminant $\cD_{n}$ of the derivative of a polynomial $p$
of degree $n$ is the sum of squares of polynomials in the differences of the roots of $p.$
  \end{conjecture}

Based on our calculations and computer experiments we propose the following extension of the latter conjecture. 
We call an arbitrary graph with all edges of even multiplicity a {\em square graph}. Notice that the symmetrized graph monomial of a square graph is obviously is  a sum of squares. 

\begin{conjecture}\label{conj:extended} 
For any non-negative integer $0\leq k \leq  n-2$ the discriminant $\cD_{n,k}$ of the $k$th derivative of a polynomial $p$
of degree $n$   is a finite positive linear combination  of  the symmetrized graph monomials  where all underlying graphs are square graphs with $n$ vertices. On other words, $\cD_{n,k}$ lies in the convex cone spanned by the symmetrized graph monomials of the square graphs with $n$ vertices and $\binom{n-k}{2}$ edges.  
\end{conjecture}

 Observe  that $\deg \cD_{n,k}=(n-k)(n-k-1)$ and is, therefore, even.    We use the following agreement in our figures below. If a shown graph has fewer vertices than $n$ then we always assume that it is appended by the required number of isolated vertices. 
 The following examples support the above conjectures.

\begin{example} If $k=0$ then $\cD_{n,0}$ is proportional to $\tilde g$ where $g$ is the complete graph on $n$ vertices with all edges of multiplicity $2$. 
\end{example}

\begin{example}
For $k\ge 0$, the discriminant $\cD_{k+2,k}$ equals 
$$\frac{(k+1)!}{2}\sum_{1\le i<j\le k+2} (x_i-x_j)^2.$$
\end{example}

In other words, $\cD_{k+2,k}=\frac{(k+2)!}{2} \tilde g$ where the graph $g$ is  
given in Fig.~\ref{fig:discrgraph4} (appended with $k$ isolated vertices). 

\begin{figure}[ht!]
\centering
\includegraphics[scale=0.5]{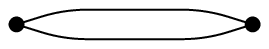}
\caption{The graph $g$ for the case $\cD_{k+2,k}$}\label{fig:discrgraph4}
\end{figure}

\begin{example}
For $k \ge 0$  we conjecture that the discriminant $\cD_{k+3,k}$ equals 

$$(k!)^3\left[
\frac{(k + 1)^3 (k + 2)(k + 6)}{72} \tilde {g_1} + 
\frac{(k + 1)^3 k(k+2)}{12} \tilde {g_2} +
\frac{(k - 1) k (k + 1)^2(k + 2)(k-2)}{96} \tilde {g_3}\right]$$
%
where the graphs $g_1,g_2$ and $g_3$ are given in Fig.~\ref{fig:discrgraphs6}. 
(This claim is verified for $k=1,\dots,12.$)\end{example}

\begin{figure}[ht!]
\centering
\includegraphics[scale=0.5]{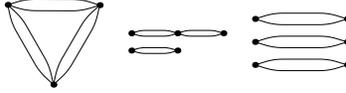}
\caption{The graphs $g_1,g_2$ and $g_3$ for the case $\cD_{k+3,k}$}\label{fig:discrgraphs6}
\end{figure}

\begin{example}
The discriminant $\cD_{5,1}$ is given by 
$$\cD_{5,1} = \frac{19}{6}\tilde{g}_1 + 14\tilde{g}_2 + 2\tilde{g}_3$$
where $g_1,g_2,g_3$ are given in Fig.~\ref{fig:d51graphs}.
\end{example}
\begin{figure}[ht!]
\centering
\includegraphics[scale=0.5]{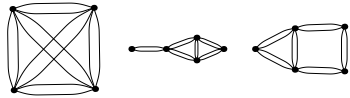}
\caption{The graphs $g_1,g_2$ and $g_3$ for the case $\cD_{5,1}$}\label{fig:d51graphs}
\end{figure}

\begin{example}
Lastly
$$\cD_{6,2} = 19200\tilde{g}_1+960\tilde{g}_2+3480\tilde{g}_3+3240\tilde{g}_4+\frac{3440}{3}\tilde{g}_5+2440\tilde{g}_6$$
where $g_1,\dots,g_6$ are given in Fig.~\ref{fig:d62graphs}. (Note that this representation as sum of graphs is by no means  unique.) 
\end{example}

The second motivation of the present study  is an interesting example of a graph whose symmetrized graph monomial is non-negative 
but not a sum of squares. Namely, the main result of \cite{lax} shows that $\tilde g$ for the graph
given in Fig.~\ref{fig:laxgraph} has this property. 

\begin{figure}[ht!]
\centering
\includegraphics[scale=0.5]{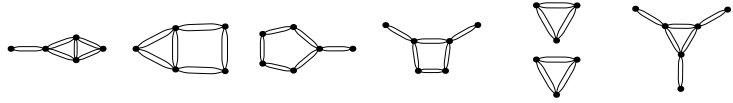}
\caption{The graphs $g_1,\dots,g_6$ for the case $\cD_{6,2}$}\label{fig:d62graphs}
\end{figure}

\begin{figure}[ht!]
\centering
\includegraphics[scale=0.5]{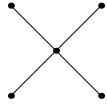}
\caption{The only 4-edged graph which yields a non-negative non-sos polynomial.}\label{fig:laxgraph}
\end{figure}

Our main computer-aided results regarding the case with 4 resp. 6-edged graphs are given below. Notice that there exist $23$  graphs with $4$ edges and $212$ graphs with $6$ edges. We say that two graphs are {\em equivalent} if their symmetrized graph monomials are non-vanishing identically and proportional.  

\begin{proposition}\label{pr:main4}
\rm{(i)}  $10$ graphs with $4$ edges have identically vanishing symmetrized graph monomial. 
\noindent \rm{(ii)} The remaining $13$ graphs are divided into 4 equivalence classes presented in Fig~\ref{fig:similar4graphs}.
\noindent
\rm{(iii)} The first two classes contain square graphs and, thus,  their symmetrized monomials are non-negative. 
\noindent
\rm{(iv)} The third graph  is  non-negative (as a positive linear combination of the Lax graph and a square graph). Since it effectively depends only on three variables, it is  SOS, see \cite{Hil}.
\noindent
\rm{(v)} The last graph is the Lax graph which is thus  the only non-negative  graph with $4$ edges  not being a  SOS. 

\begin{figure}[ht!]
\centering
\includegraphics[scale=0.4]{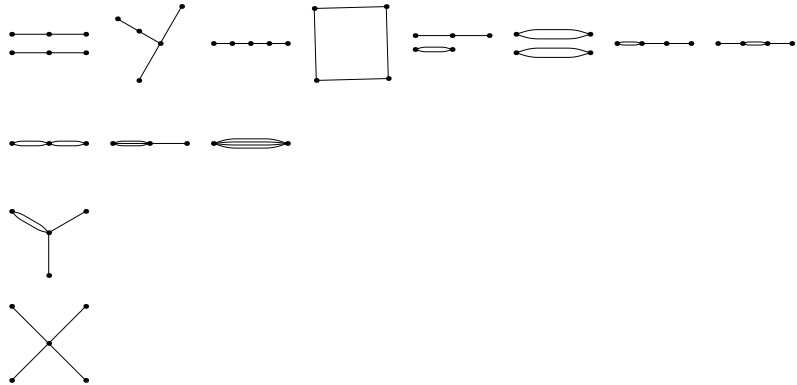}
\caption{$4$ equivalence classes of the 13  graphs with $4$ edges, whose symmetrized graph monomials do not vanish identically.}\label{fig:similar4graphs}
\end{figure}
\end{proposition}

\begin{proposition}\label{pr:main5}
\rm {(i)} $102$ graphs with $6$ edges  have identically vanishing symmetrized graph monomial.  
\rm {(ii)} The remaining $110$ graphs are divided into 27 equivalence classes. 
\rm{(iii)} $12$ of these classes can be expressed as non-negative linear combinations of square graphs, i.e. lie in the convex cone spanned by the square graphs. \rm{(iv)} Of the remaining 15 classes, symmetrized graph monomial of $7$ of them   change sign. 
\rm {(v)} Of the remaining 8 classes (which are presented on Fig.~7)  the first 5 are sums of squares, given as matrix representations in the Appendix below. (Notice however that these symmetrized graph monomials do not lie in the convex cone spanned by the square graphs.) 
\rm{(vi)} The last 3 classes contain all non-negative graphs with $6$ edges,  which are not SOS and, therefore, give new examples 
of graphs a'la Lax.

\end{proposition}

\begin{figure}[ht!]
\centering
\includegraphics[width=0.75\textwidth]{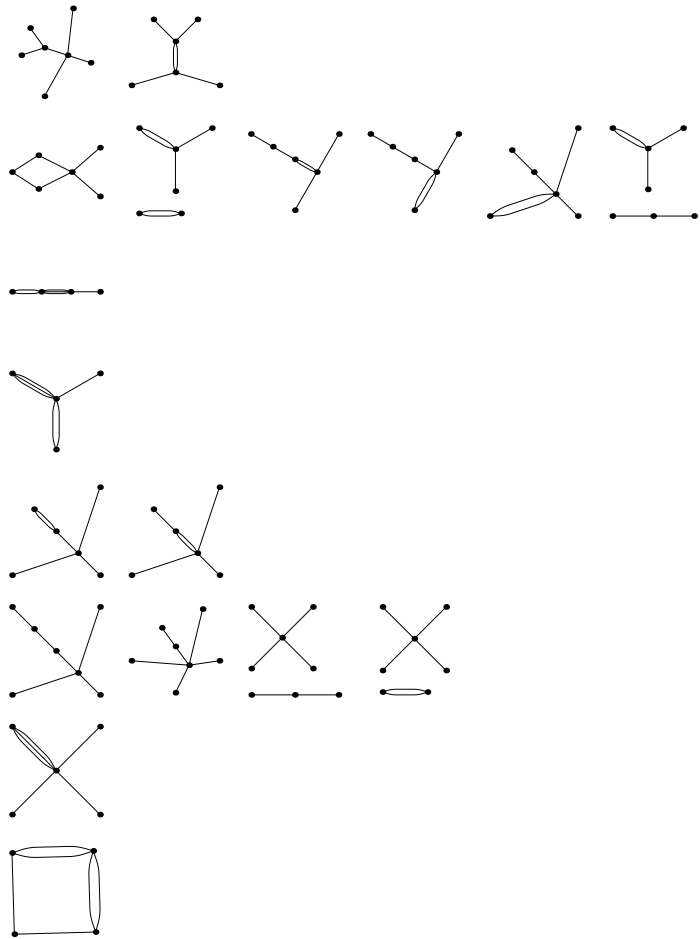}
\caption{$8$ equivalence classes of all non-negative graphs with $6$ edges.}\label{fig:similar6graphs}
\end{figure}

It is classically known that for any given number of vertices $n$ and edges $d$,
the linear span of the symmetrized graph monomials coming from all graphs with $n$ vertices and $d$ edges
coincides with the linear space $\PST_{n,d}$ of all symmetric translation-invariant polynomials of degree $d$ in $n$ variables.

We say that a pair $(n,d)$ is stable if $n\geq 2d,$ and for stable $(n,d),$ 
we suggests a natural basis in $\PST_{n,d}$ of symmetrized graph monomials which seems to be new,
proved in Lemma \ref{lm:partitioncoefficient}, Corollary \ref{cor:partitionsarebasis}.

In the case of even degree, there is a second basis in $\PST_{n,d}$ of symmetrized graph monomials
consisting of only square graphs, see Lemma \ref{lm:lowertriangular} and Corollary \ref{cor:squarebasis}.

\medskip 
Notice that  translation invariant 
symmetric polynomials appeared also in the early 1970's in the study of integrable $N$-body problems in mathematical 
physics (apparently)  see the famous paper of F.~Calogero \cite{Cal}. 
A few much more recent publications related to the ring of such polynomials in connection 
with the investigation of multi-particle interactions and the quantum Hall effect were 
printed since then, see e.g. \cite{SRC}, \cite{Lipt}. 
In particular, the ring structure and the dimensions of the homogeneous components of this ring were calculated. 
In particular, it was shown in \S~IV of \cite{SRC} and \cite{Lipt} that the ring of translation 
invariant symmetric polynomials (with integer coefficients) in $x_1,\dots,x_n$ is 
isomorphic as a graded ring to the polynomial ring $\Z[e_2,\dots,e_n]$ where $e_i$ 
stands for the $i$-th elementary symmetric function in $x_1,\dots,x_n$. 
From this fact one can easily show that the dimension of its $d$-th homogeneous component 
equals the number of distinct partitions of $d$ where each part is strictly bigger 
than $1$ and the number of parts is at most $n$. 
Several natural linear bases were also suggested for each such homogeneous component, see (29) in \cite{SRC} and \cite {Lipt}. 
It seems that the authors of the latter papers were unaware of the mathematical developments in this field related to graphs. 


\medskip 
\noindent
{\em Acknowledgements.} We are sincerely grateful  to Professors A.~Abdesselam,  F.~Sottile, B.~Sturmfels for discussions and important references and to Doctor P.~Rostalski  for his assistance with computer-aided proof of the fact that certain symmetrized graph monomials are SOS. 


\section{Some generalities about symmetrized graph monomials}

\begin{definition}
Let $g$ be a directed graph with $d$ edges and vertices $v_1,v_2,\dots,v_n.$
Let $\alpha=(\alpha_1,\dots,\alpha_d)$ be an integer partition of $d.$
A \emph{partition-coloring} of $g$ with $\alpha$ is an assignment of
colors to the edges and vertices of $g$ satisfying the following:
\begin{itemize}
\item For each color $i,$ $1\leq i \leq d,$ we paint the vertex $v_j$ and  $\alpha_i$ edges connected to $v_j$ with the color $i.$
\item Each edge of $g$ is painted with exacly one color.
\item Each vertex is painted at most once.
\end{itemize}
An edge is \emph{odd-colored} if it has color $j$ and is directed to a vertex with the same color.
The coloring is said to be negative if there is an odd number of odd-colored edges in $g,$ and positive otherwise.
\end{definition}

\begin{definition}
Given a polynomial $P(\xvec)$ and a multi-index $\alpha = (\alpha_1,\dots,\alpha_n),$
we use the notation $Coeff_\alpha(P(\xvec))$ to denote the coefficient in front of $\xvec^\alpha$ in $P(\xvec).$
\end{definition}
Note that we may view $\alpha$ as a partition of the sum of the indices. 

\begin{lemma}\label{lm:symcoloring}
Let $g$ be a directed graph with $d$ edges and vertices $v_1,v_2,\dots,v_n.$
Then $Coeff_\alpha(\tilde{g})$ is given by the number of 
positive partition-colorings of $g$ with $\alpha$ minus the number of negative partition-colorings. 
\end{lemma}
\begin{proof}
See \cite[Lemma 2.3]{sabidussi2}. 
\end{proof}

\subsection{Bases for $\PST_{n,d}$}

It is known that the dimension of $\PST_{n,d}$ with $n \geq 2d$ is given by 
the number of integer partitions of $d$ where each part is at least of size 2.
Such integer partition will be called a 2-partition.

To each 2-partition $\alpha = (\alpha_1,\alpha_2,\dots,\alpha_d)$, $\alpha_i \neq 1,$ we associate the following graph $b_\alpha$: 
For each $\alpha_i\geq 2,$ we have a connected component of $b_\alpha$
consisting of a center vertex, connected to $\alpha_i$ other vertices,
with the edges directed away from the center vertex. 
Since $\alpha$ is an integer partition of $d,$ it follows that $b_\alpha$ has exacly $d$ edges.
This type of graph will be called a \emph{partition graph}.

\begin{lemma}
Let $P(\xvec)$ be a polynomial.
Then
$$Sym_{(\xvec,\yvec)} P = \sum_{i=0}^{|\yvec|} \sum_{\stackrel{\sigma \subseteq \yvec}{|\sigma|=i} } \sum_{\stackrel{\tau \subseteq \xvec}{|\tau|=|\xvec|-i} } Sym_{(\tau \cup \sigma) } P.$$
\end{lemma}
\begin{proof}
This is standard, by a straightforward combinatorical argument.
\end{proof}
\begin{corollary}
If $Sym_\xvec P$ is non-negative, then $Sym_{(\xvec,\yvec)} P$ is non-negative.
\end{corollary}
\begin{corollary}
If $Sym_\xvec P$ is a sum of squares, then $Sym_{(\xvec,\yvec)} P$ is a sum of squares.
\end{corollary}
\begin{corollary}\label{cor:symrelation}
If $\sum_i \lambda_i Sym_\xvec P_i = 0$ then $\sum_i \lambda_i Sym_{(\xvec,\yvec)} P_i = 0.$
\end{corollary}

\subsection{Partition graphs}

We will use the notation 
that every symmetric polynomial $\tilde g$ associated with a graph on $d$ edges is symmetrized over $2d$ variables.
Corollary \ref{cor:symrelation} says that if a relation holds for symmetrizations in $2d$ variables, 
it will also hold for $2d+k$ variables, and therefore, 
each relation derived in this section also holds for $2d+k$ variables.

\begin{lemma}\label{lm:partitioncoefficient}
Let $b_\alpha$ be a partition graph with $d$ edges, $\alpha = (\alpha_1,\dots,\alpha_d),$ 
and let $\beta = (\beta_1,\beta_2,\dots,\beta_d)$ be a 2-partition.

Then $Coeff_\beta(\tilde{b_\alpha})$ is
$$
\begin{cases}
0 \mbox{ if } \beta \neq \alpha \\
\prod_{j=2}^d \#\{i |\alpha_i = j\}! \mbox{ if } \beta = \alpha
\end{cases}
$$
\end{lemma}
\begin{proof}
We will try to color the graph $b_\alpha$ with $\beta:$

Since $\beta_i \neq 1,$ we may only color the center vertices of $b_\alpha.$
Hence, all edges in each component of $b_\alpha$ must have the same color as the center vertex.
It is clear that such coloring is impossible if $\alpha \neq \beta.$
If $\alpha = \beta,$ we see that each coloring has positive sign, 
since only center vertices are colored and all connected edges are directed outwards.

The only difference between two colorings must be the assignment of the colors to the center vertices.
Hence, components with the same size can permute colors, which yields
$$ \prod_{j=2}^d \#\{i |\alpha_i = j\}! $$
number of ways to color $g$ with the partition $\alpha_1,\dots,\alpha_p.$
\end{proof}
\begin{corollary}\label{cor:partitionsarebasis}
It follows that all partition graphs yield linearly independent polynomials,
since each partition graph $b_\alpha$ unique contributes with the monomial $\xvec^\alpha.$
The number of partition graphs on $d$ edges equals the dimension of $\PST_{d,n},$
and must therefore span the entire vector space.
\end{corollary}

\subsection{Square graphs}

We will use the notation $\alpha = (\alpha_1,\dots,\alpha_q | \alpha_{p+1},\dots,\alpha_p)$ 
to denote a partition where $\alpha_1,\dots,\alpha_q$ are the odd parts in decreasing order,
and $\alpha_{q+1},\dots,\alpha_p$ are the even parts in decreasing order. 
Parts are allowed to be equal to 0, so that $\alpha$ can be used as multi-index over $p$ variables.

Now we define a second type of graphs that we associate with 2-paritions of even integers:

Let $\alpha = (\alpha_1,\alpha_2,\dots,\alpha_k | \alpha_{k+1},\dots,\alpha_d)$, $\alpha_i \neq 1,$ be a 2-partition of $d.$ 
Since this is a partition of an even integer, $k$ must be even.

For each even $\alpha_i\geq 2,$ we have a connected component of $h_\alpha$
consisting of a center vertex, connected to $\alpha_i/2$ other vertices,
with the edges directed away from the center vertex, and with multiplicity 2. 

For each pair $\alpha_{2j-1},\alpha_{2j}$ of odd parts, $j=1,2,\dots,\frac{k}{2}$ we 
have a connected component consisting of two center vertices $v_{2j-1}$ and $v_{2j},$ 
such that $c_i$ is connected to $\lfloor \alpha_i/2 \rfloor$ other vertices for $i=2j-1,2j$ with edges of multiplicity 2,
and the center vertices are connected with a double edge. This type of component will be called a \emph{glued component}.

Thus, each edge in $h_\alpha$ has multiplicity 2, and the number of edges, counting multiplicity, is $d.$
This type of graph will be called a \emph{square graph}. Note that $\tilde{h}_\alpha(\xvec)$ is a sum of squares. 

\begin{figure}[ht!]
\centering
\includegraphics[width=0.75\textwidth]{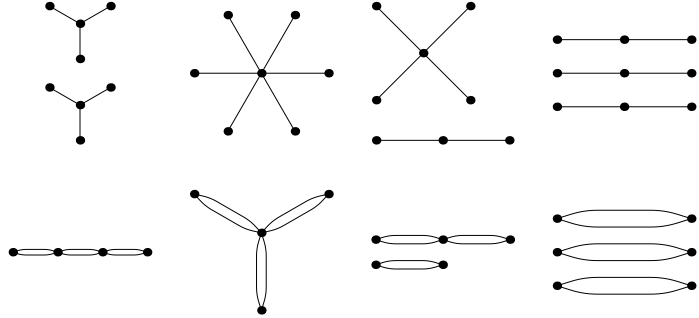}
\caption{A  base of partition graphs and a base of square graphs in the stable case  with 6 edges.}
\end{figure}

\begin{lemma}\label{lm:squarecoefficient}
Let $h_\alpha$ be a square graph where $\alpha=(\alpha_1,\dots,\alpha_p).$
Then
$$Coeff_\alpha(\tilde{h}) = (-1)^{\frac12 \#\{i |\alpha_i \equiv_2 1\}} 2^{\#\{i |\alpha_i = 2\}} \prod_{j=2}^n \#\{i |\alpha_i = j\}! $$
\end{lemma}
\begin{proof}
Similarly to Lemma \ref{lm:partitioncoefficient}, 
it is clear that a coloring of $h$ with $p$ colors 
require that each center vertex is painted.

The center vertex of a component with only two vertices is not uniquely determined,
so we have $ 2^{\#\{i |\alpha_i = 2\}}$ choices of centers.

It is clear that each glued component contributes with exacly one odd edge for every coloring,
and the sign is therefore the same for each coloring.
The number of glued components are precisely $\frac12 \#\{i |\alpha_i \equiv_2 1\}.$ 

Lastly, we may permute the colors corresponding to center vertices with the same degree.
These observations together yields the formula
$$(-1)^{\frac12 \#\{i |\alpha_i \equiv_2 1\}} 2^{\#\{i |\alpha_i = 2\}} \prod_{j=2}^n \#\{i |\alpha_i = j\}! .$$
\end{proof}

Define a total order on 2-partitions as follows:
\begin{definition}
Let $\alpha = (\alpha_1,\dots,\alpha_p | \alpha_{p+1},\dots,\alpha_q)$ and 
$\alpha' = (\alpha'_1,\dots,\alpha'_{p'} | \alpha'_{p'+1},\dots,\alpha_{q'})$ be 2-partitions.
We say that $\alpha \prec \alpha'$ if $\alpha_i = \alpha'_i$ 
for $i=1,\dots,j-1,$ $j\geq 1$ and one of the following holds:
\begin{itemize}
\item $\alpha_j > \alpha'_j$ and $\alpha_j = \alpha'_j \mod 2$
\item $\alpha_j$ is odd and $\alpha'_j$ is even.
\end{itemize}
This generalizes to $\alpha \preceq \alpha' \Leftrightarrow \alpha \prec \alpha' \mbox{ or } \alpha = \alpha'.$ 
\end{definition}

\begin{lemma}\label{lm:lowertriangular}
Let $h_\alpha$ be a square graph. Then we may write
\begin{eqnarray}\label{eqn:lincombsq}
\tilde{g_\alpha} = \sum_{\beta} \lambda_\beta \tilde{b}_\beta, \quad b_\beta \mbox{ is a partition graph},
\end{eqnarray}
where $\lambda_\beta=0$ if $\beta \prec \alpha.$
\end{lemma}
\begin{proof}
Let $\alpha = (\alpha_1,\dots,\alpha_q | \alpha_{q+1}\dots,\alpha_d)$
and let $\beta = (\beta_1,\dots,\beta_r | \beta_{r+1},\dots,\beta_d),$ 
with $\beta \prec \alpha.$ 
Consider equation \eqref{eqn:lincombsq} and apply $Coeff_\beta$ on both sides. 
Lemma \ref{lm:partitioncoefficient} implies 
$$Coeff_\beta(\tilde{h}_\alpha) = \lambda_\beta \cdot C_\beta, \text{ where } C_\beta>0.$$
It suffices to show that there is no partition-coloring of $h_\alpha$ with $\beta$ if $\beta \prec \alpha,$ 
since this implies $\lambda_\beta=0.$

We now have three cases to consider:

\textbf{Case 1:} $\alpha_i=\beta_i$ for $i=1,\dots,j-1$ and 
$\beta_j > \alpha_j$ where $\alpha_j$ and $\beta_j$ are either both odd or both even.

We must paint a center vertex and $\beta_j$ connected edges, since $\beta_j>\alpha_j\geq 2.$

There is no vacant center vertex in $g_\alpha$ 
with degree at least $\beta_j,$ 
all such centers have already been painted with the colors $1,\dots,j-1.$
Hence a coloring is impossible in this case.

\textbf{Case 2:} $\alpha_i=\beta_i$ for $i=1,\dots,j-1$, $\beta_j$ is odd and $\alpha_j$ is even.
This condition impies that $r<q.$ 

Every component of $h_\alpha$ has an even number of edges,
and only vertices with degree at least three can be colored with an odd color.
Therefore, glued components must be colored with exacly zero or two odd colors,
and non-glued component must have an even number of edges of each present color.
This implies that a coloring is only possible if $r\leq q,$ but this is not true in the considered case.

Hence, there is no coloring of $h_k$ with the colors given by $\beta,$ and therefore,
the coefficient in front of $\xvec^\beta$ is 0 in $\tilde{g}_k,$ implying $\lambda_\beta =0.$
\end{proof}

\begin{corollary}\label{cor:squarebasis}
The polynomials obtained from the square graphs with $d$ edges is a basis for $\PST_{d,n},$ if $d$ is even.
\end{corollary}
\begin{proof}
Let $\alpha_1 \prec \dots \prec \alpha_k$ be the 2-partitions of $d.$
Since $\tilde{b}_{\alpha_1},\dots,\tilde{b}_{\alpha_k}$ is a basis, there is a uniquly determined matrix $M$ such that
$$(\tilde{h}_{\alpha_1},\dots,\tilde{h}_{\alpha_k})^T = M(\tilde{b}_{\alpha_1},\dots,\tilde{b}_{\alpha_k})^T.$$
Lemma \ref{lm:lowertriangular} implies that $M$ is lower-triangular. 
Lemma \ref{lm:partitioncoefficient} and Lemma \ref{lm:squarecoefficient} 
implies that the entry at $(\alpha_i,\alpha_i)$ in $M$ is given by
$$(-1)^{\frac12 \#\{j |\alpha_{ij} \equiv_2 1\}} 2^{\#\{j |\alpha_{ij} = 2\}}$$
which is non-zero. Hence $M$ has an inverse and the square graphs is a basis.
\end{proof}

%
%
%
%
%
%

\section{Final remarks} 
Some obvious challenges related to this project are as follows.  

\medskip
\noindent
{\bf 1.}� Prove Conjectures~1 and ~2.

\medskip
\noindent
{\bf 2.}� Describe the boundary of the convex cone spanned by all square graphs with a given number of (double) edges and vertices.

\medskip
\noindent
{\bf 3.} Find more examples of graphs a'la Lax.
\medskip 

\section{Appendix}

Here we give a SOS presentation of the symmetrized graph monomial for the first $5$ classes given in Fig.~\ref{fig:similar6graphs}.
Each row yield the same polynomial, up to a constant. 
The symmetrized graph monomial from row $i$ is a constant multiple of the polynomial $v_i Q_i v_i^T,$
where $v_i$ is the coefficient vector and $Q_i$ is the corresponding symmetric positive semi-definite matrix, given below.
This certifies that the first 5 classes of graphs are sum of squares.

By using Matlab together with Yalmip, one may verify that the last three classes cannot be expressed as sums of squares.
It is relatively straightforward to verify that the polynomials indeed are non-negative, 
using methods similarly to \cite{lax}.

\begin{figure}[ht!]
\centering
\includegraphics[scale=0.3]{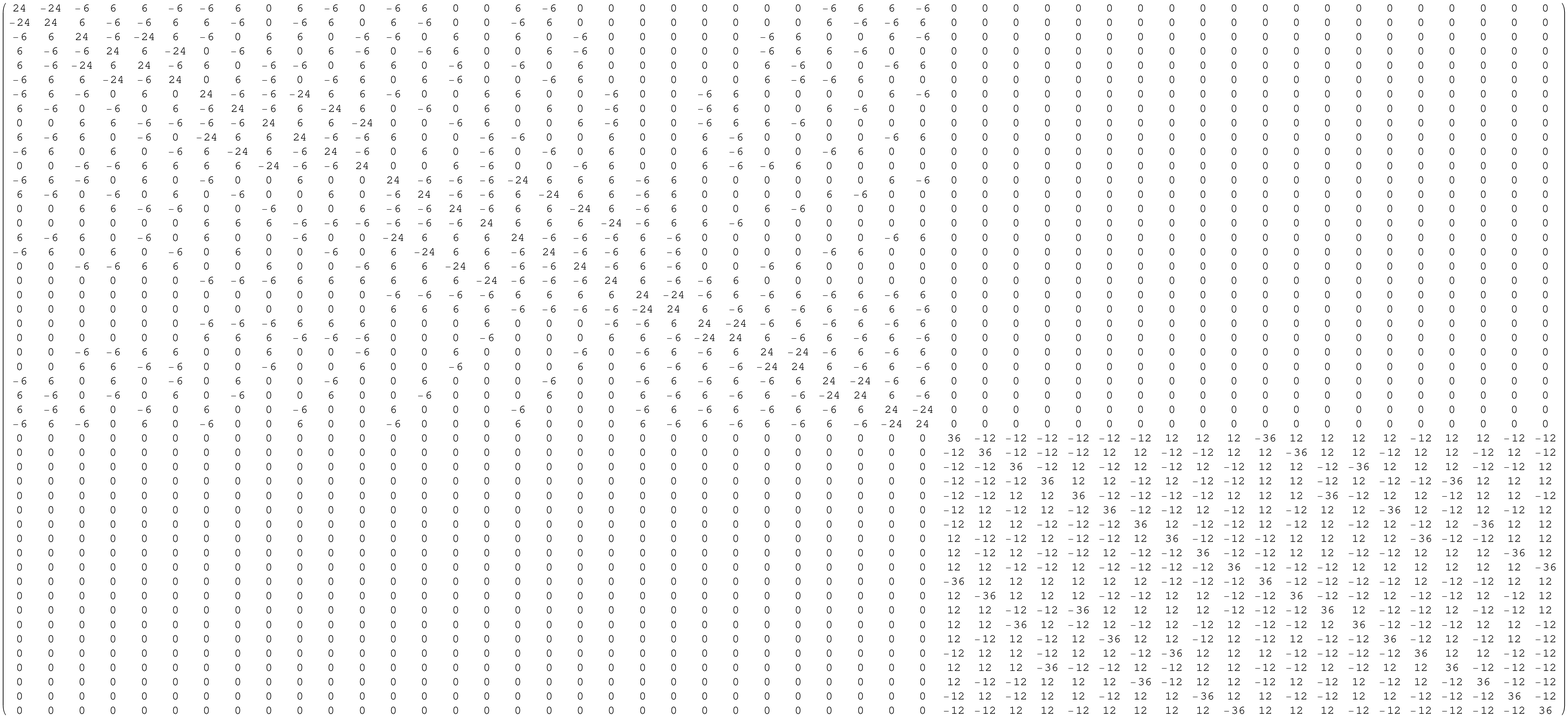}
\caption{$Q_1$}
\end{figure}

\begin{figure}[ht!]
\centering
\includegraphics[scale=0.3]{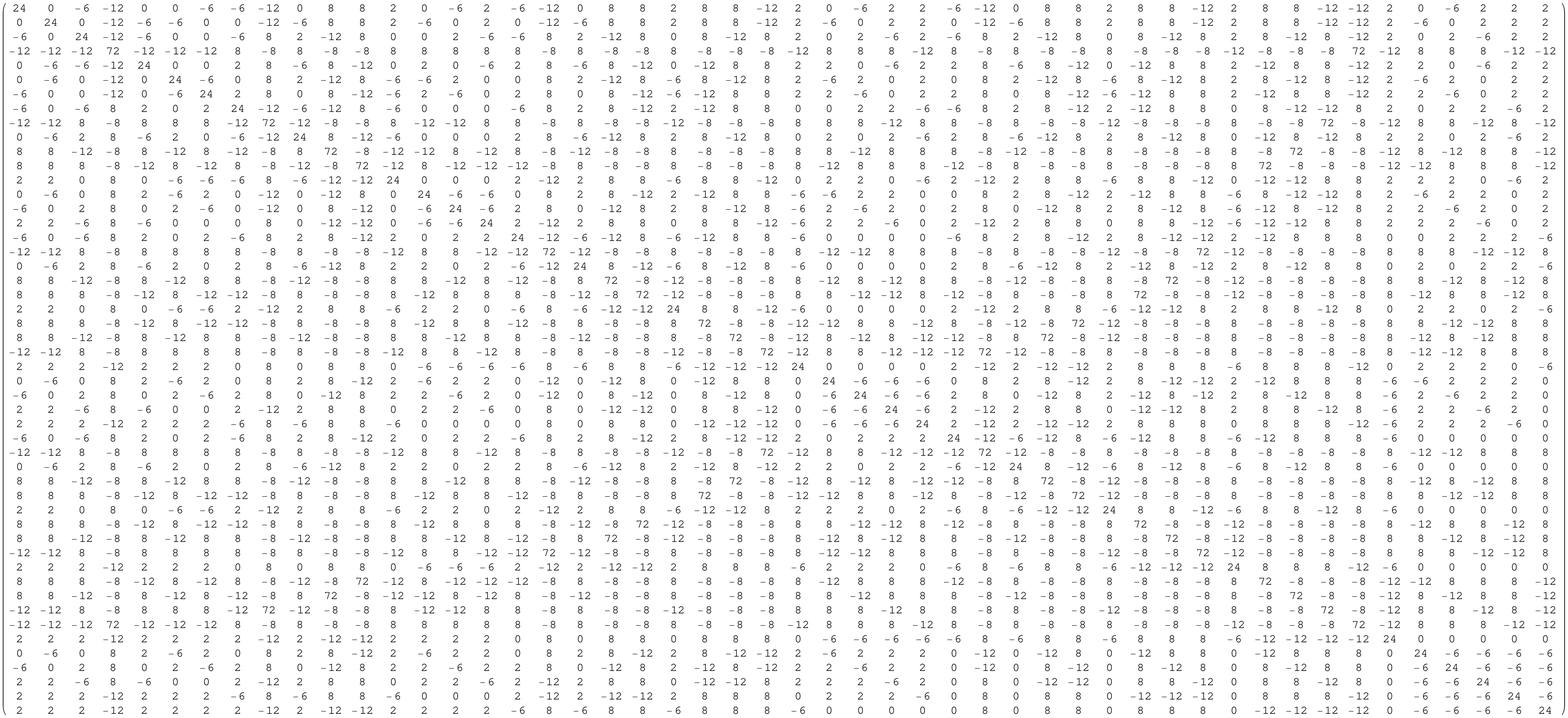}
\caption{$Q_{2}$}
\end{figure}

\begin{figure}[ht!]
\centering
\includegraphics[scale=0.3]{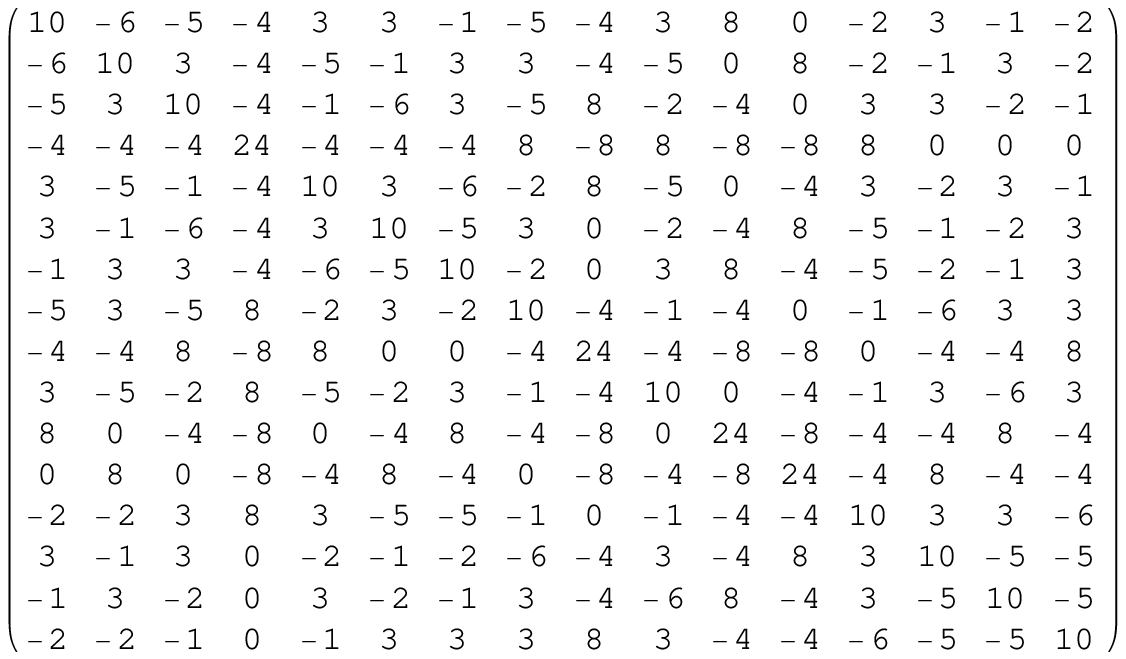}
\caption{$Q_{3}$}
\end{figure}

\begin{figure}[ht!]
\centering
\includegraphics[scale=0.3]{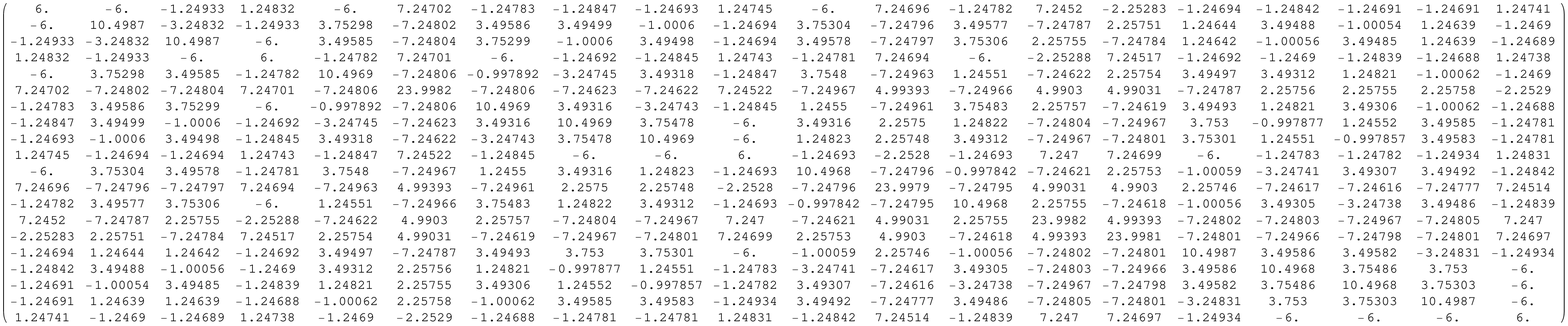}
\caption{$Q_{4}$}
\end{figure}

\begin{figure}[ht!]
\centering
\includegraphics[scale=0.3]{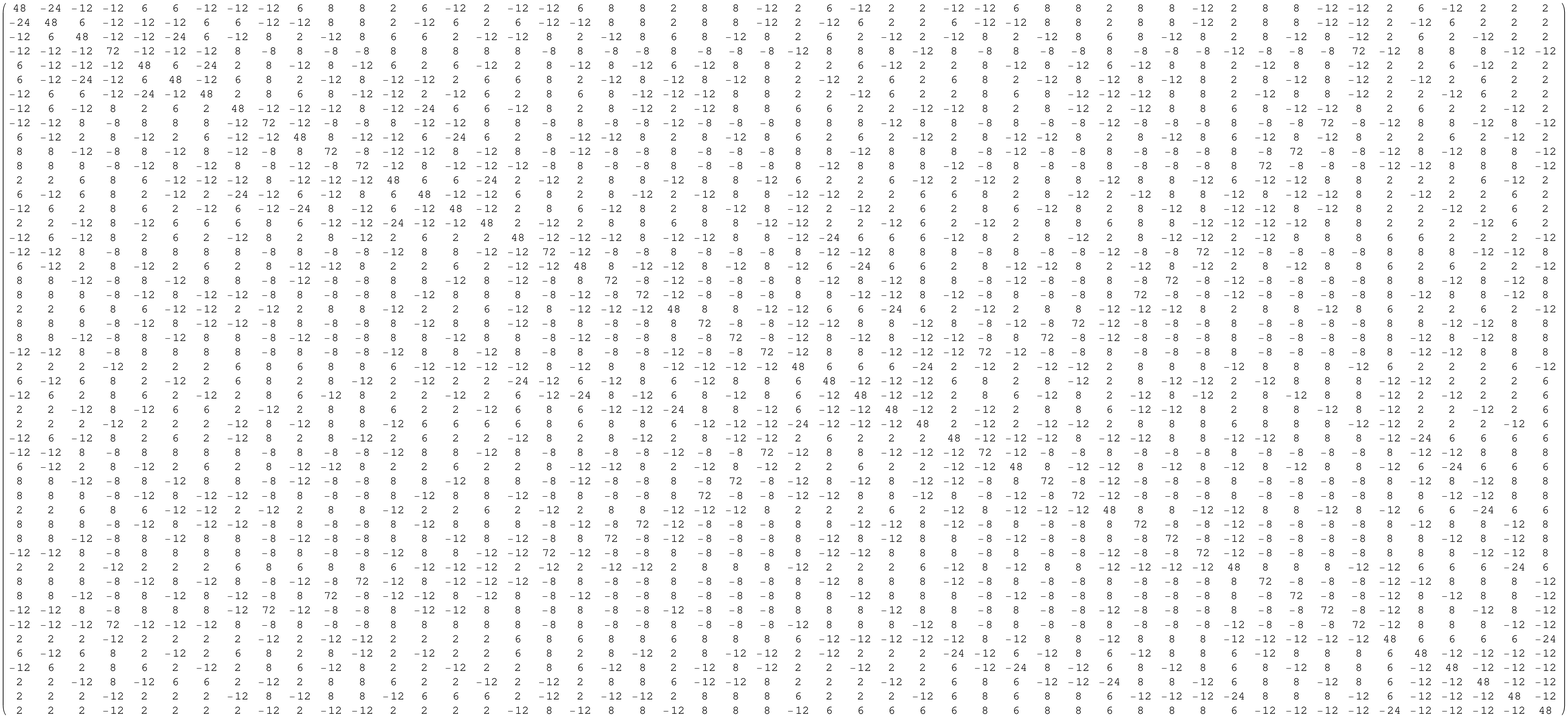}
\caption{$Q_{5}$}
\end{figure}

\begin{align*}
v_1=\{&x_1^2 x_2, x_1 x_2^2, x_1^2 x_3, x_2^2 x_3,x_1 x_3^2,x_2 x_3^2,x_1^2 x_4,x_2^2 x_4,x_3^2 x_4,x_1 x_4^2,x_2 x_4^2,x_3 x_4^2,\\
     &x_1^2 x_5,x_2^2x_5,x_3^2 x_5,x_4^2 x_5,x_1 x_5^2,x_2 x_5^2,x_3 x_5^2,x_4 x_5^2,x_5 x_6^2,x_5^2 x_6,x_4 x_6^2,x_4^2 x_6,\\
     &x_3 x_6^2,x_3^2 x_6,x_2 x_6^2,x_2^2x_6,x_1 x_6^2,x_1^2 x_6,x_1 x_2 x_3,x_1 x_2 x_4,x_1 x_3 x_4,x_2 x_3 x_4,x_1 x_2 x_5, \\
     &x_1 x_3 x_5,x_2 x_3x_5,x_1 x_4 x_5,x_2 x_4 x_5,x_3 x_4x_5,x_4 x_5 x_6,x_3 x_5 x_6,x_3 x_4 x_6,x_2 x_5 x_6,x_2 x_4 x_6, x_2 x_3 x_6, \\
     &x_1 x_5 x_6,x_1 x_4 x_6,x_1 x_3 x_6,x_1 x_2 x_6\}
\end{align*}

\begin{align*}
v_{2}=\{&x_5 x_6^2,x_5^2 x_6,x_4 x_6^2,x_4 x_5 x_6,x_4 x_5^2,x_4^2 x_6,x_4^2 x_5,x_3 x_6^2,x_3 x_5 x_6,x_3 x_5^2,x_3 x_4 x_6,\\
          &x_3 x_4 x_5,x_3x_4^2,x_3^2 x_6,x_3^2 x_5,x_3^2 x_4,x_2 x_6^2,x_2 x_5 x_6,x_2 x_5^2,x_2 x_4 x_6,x_2 x_4 x_5,x_2 x_4^2,\\
          &x_2 x_3 x_6,x_2 x_3 x_5,x_2 x_3 x_4,x_2x_3^2,x_2^2 x_6,x_2^2 x_5,x_2^2 x_4,x_2^2 x_3,x_1 x_6^2,x_1 x_5 x_6,x_1 x_5^2,\\
          &x_1 x_4 x_6,x_1 x_4 x_5,x_1 x_4^2,x_1 x_3 x_6,x_1 x_3 x_5,x_1x_3 x_4,x_1 x_3^2,x_1 x_2 x_6,x_1 x_2 x_5,x_1 x_2 x_4,\\
          &x_1 x_2 x_3,x_1 x_2^2,x_1^2 x_6,x_1^2 x_5,x_1^2 x_4,x_1^2 x_3,x_1^2 x_2\}
\end{align*}

\begin{align*}
v_{3}=\{&x_3 x_4^2,x_3^2 x_4,x_2 x_4^2,x_2 x_3 x_4,x_2 x_3^2,x_2^2 x_4,x_2^2 x_3,x_1 x_4^2,x_1 x_3 x_4,x_1 x_3^2,x_1 x_2 x_4,\\
          &x_1 x_2 x_3,x_1x_2^2,x_1^2 x_4,x_1^2 x_3,x_1^2 x_2\}
\end{align*}

\begin{align*}
v_{4}=\{&x_4^3,x_3 x_4^2,x_3^2 x_4,x_3^3,x_2 x_4^2,x_2 x_3 x_4,x_2 x_3^2,x_2^2 x_4,x_2^2 x_3,x_2^3,x_1 x_4^2,x_1 x_3 x_4,x_1 x_3^2,\\
          &x_1 x_2 x_4,x_1x_2 x_3,x_1 x_2^2,x_1^2 x_4,x_1^2 x_3,x_1^2 x_2,x_1^3\}
\end{align*}

\begin{align*}
v_{5}=\{&x_5 x_6^2,x_5^2 x_6,x_4 x_6^2,x_4 x_5 x_6,x_4 x_5^2,x_4^2 x_6,x_4^2 x_5,x_3 x_6^2,x_3 x_5 x_6,x_3 x_5^2,x_3 x_4 x_6,\\
           &x_3 x_4 x_5,x_3x_4^2,x_3^2 x_6,x_3^2 x_5,x_3^2 x_4,x_2 x_6^2,x_2 x_5 x_6,x_2 x_5^2,x_2 x_4 x_6,x_2 x_4 x_5,x_2 x_4^2,\\
           &x_2 x_3 x_6,x_2 x_3 x_5,x_2 x_3 x_4,x_2x_3^2,x_2^2 x_6,x_2^2 x_5,x_2^2 x_4,x_2^2 x_3,x_1 x_6^2,x_1 x_5 x_6,x_1 x_5^2,\\
           &x_1 x_4 x_6,x_1 x_4 x_5,x_1 x_4^2,x_1 x_3 x_6,x_1 x_3 x_5,x_1x_3 x_4,x_1 x_3^2,x_1 x_2 x_6,x_1 x_2 x_5,x_1 x_2 x_4,\\
           &x_1 x_2 x_3,x_1 x_2^2,x_1^2 x_6,x_1^2 x_5,x_1^2 x_4,x_1^2 x_3,x_1^2 x_2\}
\end{align*}



\end{document}